\documentclass{amsart}

\usepackage{amsmath}
\usepackage{array}
\usepackage{lipsum}
\usepackage{amsfonts}
\usepackage{algorithmic}
\usepackage{mathtools}
\usepackage{float}
\usepackage{cleveref}
\usepackage{tikz}
\usepackage{bm}
\usepackage{graphicx}
\usepackage{subfigure}
\usepackage{epstopdf}
\ifpdf
\DeclareGraphicsExtensions{.eps,.pdf,.png,.jpg}
\else
\DeclareGraphicsExtensions{.eps}
\fi

\newtheorem{theorem}{Theorem}[section]
\newtheorem{lemma}[theorem]{Lemma}
\newtheorem{problem}{Problem}[section]
\newtheorem{proposition}{Proposition}[section]

\theoremstyle{definition}

\theoremstyle{remark}
\newtheorem{remark}[theorem]{Remark}

\numberwithin{equation}{section}

\begin{document}

\title[Riemann boundary value problems for the Chaplygin gas]{Riemann boundary value problems for the Chaplygin gas outside a convex cornered wedge}

%    Information for first author
\author{Bingsong Long}
%    Address of record for the research reported here
\address{School of Mathematics and Statistics, Huanggang Normal University, Hubei 438000, China}
%    Current address
%\curraddr{School of Mathematical Sciences, Fudan University, Shanghai, 200433, People's Republic of China}
\email{longbingsong@hgnu.edu.cn}
%    \thanks will become a 1st page footnote.
%\thanks{\textbf{Funding}: This work was supported in part by  the National Natural Science Foundation of China (61672531) and China Scholarship Council (201506100083).}

%    Information for second author
%\author{AiFang Qu}
%\address{Department of Mathematics, Shanghai Normal University, Shanghai 200234, China}
%\email{afqu@shnu.edu.cn}

%    General info

\date{\today}

%\dedicatory{This paper is dedicated to our advisors.}


\begin{abstract}
We consider two-dimensional Riemann boundary value problems of Euler equations for the Chaplygin gas with two piecewise constant initial data outside a convex cornered wedge. In self-similar coordinates, when the flow at the wedge corner is subsonic, this problem can be reformulated as a boundary value problem for nonlinear degenerate elliptic equations in concave domains containing a corner larger than $\pi$. It is shown that there does not exist a global Lipschitz solution for this case. We analyze the sign of the flow velocity along a certain direction, and then obtain this result by deriving a contradiction. Besides, the unique existence of the solution to the problem is established when the flow at the wedge corner is supersonic. The results obtained here are also valid for the problem of shock diffraction by a convex cornered wedge.
\end{abstract}

\maketitle

%\section*{This is an unnumbered first-level section head}
%This is an example of an unnumbered first-level heading.

%% The correct journal style for \specialsection is all uppercase; a known bug
%% in amsart.cls prevents this, so input must be uppercase until it is fixed.
%\specialsection*{This is a Special Section Head}
%\specialsection*{THIS IS A SPECIAL SECTION HEAD}
%This is an example of a special section head%
%%%%%%%%%%%%%%%%%%%%%%%%%%%%%%%%%%%%%%%%%%%%%%%%%%%%%%%%%%%%%%%%%%%%%%%%
%\footnote{Here is an example of a footnote. Notice that this footnote
%text is running on so that it can stand as an example of how a footnote
%with separate paragraphs should be written.
%\par
%And here is the beginning of the second paragraph.}%
%%%%%%%%%%%%%%%%%%%%%%%%%%%%%%%%%%%%%%%%%%%%%%%%%%%%%%%%%%%%%%%%%%%%%%%%嗯

\section{Introduction}\label{sec:1}
   The mathematical study of Riemann problems dates back to 1860 by the work of Riemann. These problems have been well studied in the one-dimensional regime (cf. \cite{Br99,Chang89,La63}). But for the two-dimensional regime, there has been no rigorous mathematical theory because the interactions of different nonlinear waves, mixed hyperbolic–elliptic type equations and the singularity of corners must be considered. To date, all efforts to mathematically analyze the two-dimensional Riemann problems for the polytropic gas have focused only on some special cases; see, for example, \cite{Chen97,RK84,LiZ09,ShengZ99,Wagner83,ZhangZ90}. In contrast, based on the properties of the Chaplygin gas, the study of these problems has yielded satisfactory results; see \cite{Bren05,CQu12,GuoSZ10,SL16,Serre09}.

    It is known that the Riemann problems only contain initial data, which are always called Riemann initial value problems. In practice, many physical problems often have some restrictions on the boundary. Therefore, Riemann boundary value problems, which are initial boundary value problems, are well worth being studied in the two-dimensional regime. Chen and Qu \cite{CQ12} established the existence of global self-similar solutions to two-dimensional Riemann boundary value problems of the Chaplygin gas outside a concave cornered wedge. In this paper, we further study these problems of the Chaplygin gas outside a convex cornered wedge.         
    
    The two dimensional full compressible Euler equations are
    \begin{equation}\label{eq:Euler system}
		\begin{cases}
			\partial_t\rho+\partial_x(\rho u)+\partial_y(\rho v)=0,\\
			\partial_t(\rho u)+\partial_x(\rho u^2+p)+\partial_y(\rho uv)=0,\\
                \partial_t(\rho v)+\partial_x(\rho uv)+\partial_y(\rho v^2+p)=0,
		\end{cases}
	\end{equation}
  where~$\rho>0, (u,v)$~are unknowns and represent the gas density and velocity, respectively; $p$ is the gas pressure. The state equation for the Chaplygin gas is \cite{Chap04}:
    \begin{equation}\label{eq:Chaplygin gas}
	  p(\rho)=A\Big(\frac{1}{\rho_{*}}-\frac{1}{\rho}\Big),
    \end{equation}
  with $\rho_{*}, A>0$ two constants. As a model of dark energy in the universe, the Chaplygin gas always appears in a number of cosmological theories \cite{KMP01,Popov10}. It follows from \eqref{eq:Chaplygin gas} that the sound speed of this gas is $c=c(\rho)=\sqrt{A}/\rho$. This indicates that any shock is characteristic and reversible, i.e., any rarefaction wave can be regarded as a shock with negative strength (see \cite{Serre09}). In what follows, both shocks and rarefaction waves are called pressure waves for convenience. 
  
  Now we describe our problem in details. On the $(x,y)$ plane, we give a convex cornered wedge (see \Cref{fg-pro1})
    \begin{equation}\label{wedge}
	  W:=\{(x,y):y<0, x<-y\cot\theta_0\},
    \end{equation}
  with two boundarys $\partial W_+:=\{y=-x\tan\theta_0, x\geq 0\}$ and $\partial W_-:=\{y=0, x\leq 0\}$. Then, our problem can be formulated mathematically as the following initial boundary value problem.
  
  \begin{problem}\label{prob1}
  We wish seek a solution of system \eqref{eq:Euler system}--\eqref{eq:Chaplygin gas} above the wedge \eqref{wedge} with the initial data
    \begin{equation}\label{initial condition}
		U(0,x,y)=(u,v,c)(0,x,y)=
        \begin{cases}
			(0,0,c_0)\quad& -\theta_0\leq\theta<\pi/2,\\
			(u_1,0,c_1)\quad& \pi/2<\theta\leq\pi,	
        \end{cases}
    \end{equation}  
  where $\theta=\arctan\dfrac{y}{x}$, and the slip boundary condition
	\begin{equation}\label{slip condition}
		(u,v)\cdot \bm{\nu}=0,
	\end{equation}
   where $\bm{\nu}$ is the exterior unit normal to $\partial W_+\cup\partial W_-$.
  \end{problem}

  \begin{figure}[H]
	\centering
	\begin{tikzpicture}[smooth, scale=0.8]
        %\draw [step=1,help lines] (0,0) grid (9, 7);
	\draw  [-latex] (8,3)--(9,3) node [right] {\footnotesize$x$};
        \draw  [-latex]  (8.5,2.5)--(8.5,3.5)node [above] {\footnotesize$y$};
        \draw  (4.5,3)--(4.5,6);
        \node at (4.5,3) [below] {\footnotesize$O$};
        \node at (3.5,2) {$W$};
        \node at (3,3) [below]{\footnotesize$\partial W_-$};
        \node at (6.5,1.8) [left]{\footnotesize$\partial W_+$};
        \draw  [dashed] (4.5,3)--(6.5,3);
	\draw  (4.5,3)--(8,1);
        \draw [draw=gray, fill=gray, fill opacity=0.6](4.5,3)--(8,1)--(1,1)--(1,3)--(4.5,3);
        \node at (2.5,4.5){\footnotesize$U_1=(u_1,0,c_1)$};
        \node at (6.5,4.5){\footnotesize$U_0=(0,0,c_0)$};
        \draw (5,2.75)arc (315:370:0.3);
        \node at (5.75,2.75){\footnotesize$\theta_0>0$};
	\end{tikzpicture}
	\caption{Riemann boundary value problems.} 
	\label{fg-pro1}
  \end{figure}
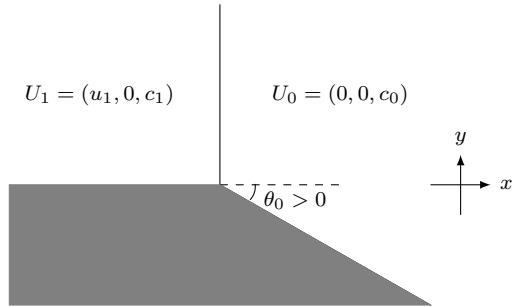

  We mainly focus on the existence of global self-similar solutions to \Cref{prob1}. There are three scenarios that need to be studied in \Cref{prob1} according to the relation between $u_1$ and $c_1$. If $u_1<c_1$, then the velocity of the flow at the origin is subsonic. For this case, under the assumption that the flow is continuous across the sonic line, we show that there does not exist a self-similar solution, which is Lipschitz at the origin $O$ (i.e. the case $(i)$ of \Cref{thm: Main2} below). If $c_1<u_1<c_0+c_1$, then the flow is supersonic at the origin, and there exists a global self-similar solution to \Cref{prob1} when the angle $\theta_0$ is small. For the critical case $u_1=c_1$, the existence of solutions is hard to establish, which is caused by the nonconvexity of the boundary of the subsonic region; see \Cref{Appendix1} for details. We leave it to the future. Moreover, by further assumptions on the initial data, \Cref{prob1} can be reduced to the problem of shock diffraction by a convex cornered wedge, which is also important problem in fluid mechanics \cite{Light49,Light50}. 

  Owing to the features of the wedge and the initial data, we can restate \Cref{prob1} in the self-similar coordinates $(\xi,\eta):=({x}/t,{y}/t)$; see \Cref{sec:2.2,sec:2.3}. Note that the corner $O$ larger than $\pi$ appears in the elliptic region for the case $u_1<c_1$. This leads to a boundary value problem for a nonlinear degenerate elliptic equation in concave domains containing such a corner point (see \Cref{prob2}). There have been many works on the degenerate boundary; see \cite{CF2018} and the references therein. However, for the corner larger than $\pi$ in concave domains, this has been studied only in some simplified models of the shock diffraction problem by convex cornered wedges, which was first considered by Lighthill \cite{Light49,Light50}. For the nonlinear wave system, the global existence of self-similar solutions was established by Kim \cite{Kim10}, and Chen et al. \cite{CDX14}. For the potential equation, Chen et al. \cite{X20} proved the nonexistence of regular solutions to this problem, where the notion of regular solutions is referred to  \cite[Definition 2.2]{X20}. For the Chaplygin gas, we obtain a similar result to \cite{X20} when \Cref{prob1} reduces to the shock diffraction problem by a convex cornered wedge (i.e. the case $(i)$ of \Cref{thm: shock diffraction} below). It should be pointed out that compared to the result in \cite{X20}, the restrictions on the solution in our result are more relaxed. More precisely, the additional condition of the regular solution in \cite[Definition 2.2]{X20} can be removed. 
  
  For concave domains containing a corner larger than $\pi$, one of the main difficult is analyzing the regularity of the solution to the elliptic equation. Note that the coefficients of the Euler equations for potential flow depend on the gradient of the potential function (see equations \eqref{eq:phi} or \eqref{eq:varphi}). Besides, the Lipschitz boundedness of the solution is a necessary condition for the potential equations \eqref{eq:phi} or \eqref{eq:varphi} to be uniformly elliptic. Therefore, to investigate the existence of self-similar solutions to \Cref{prob1}, it is essential to analyze whether the regularity of the solution reaches Lipschitz at the corner $O$. Moreover, Serre \cite{Serre09} showed that for the Chaplygin gas, if a piecewise smooth flow is irrotational and isentropic initially, then it remains so forever; that is, the flow under consideration is exactly potential flow. This fact means that once there does not exist a global Lipschitz solution for \Cref{prob2} below, we cannot expect the existence of global self-similar solutions to \Cref{prob1} even for Euler flow.   
  
  The organization of the rest of the paper as follows. \Cref{sec:2} first recalls some basic facts of the Riemann problems for the Chaplygin gas. After that, we analyze the global structures of the pressure waves in self-similar coordinates for the case $u_1<c_1$, and state main theorem of this paper as \Cref{thm: Main2}. \Cref{sec:3} further analyzes the case $u_1<c_1$. We show that a solution defined in \Cref{prob2} is monotonous with respect to the direction of $\xi$-axis, and then obtain the result of the case $(i)$ of \Cref{thm: Main2} by contradiction derivation. \Cref{sec:4} is devoted to study the case $c_1<u_1<c_0+c_1$. We establish the pressure wave structure and establish the existence of solutions for \Cref{prob1}. \Cref{sec:5} considers the problem of shock diffraction by a convex cornered wedge for the Chaplygin gas. \Cref{Appendix1} provides a preliminary analysis of the critical case $u_1=c_1$.

\section{Preliminaries}\label{sec:2} 

\subsection{Basic facts}\label{sec:2.1}   
  We give some useful known facts about the Riemann problems of the Chaplygin gas; see \cite{CQu12,Serre09} for the detailed proof. 
  
  In one-dimensional case, for the Euler equations
    \begin{equation}\label{eq:1-D Euler system}
		\begin{cases}
			\partial_t\rho+\partial_x(\rho u)=0,\\
			\partial_t(\rho u)+\partial_x(\rho u^2+p)=0,
		\end{cases}
	\end{equation}
 with initial data
    \begin{equation}\label{1-D initial condition}
		(u,c)(0,x)=
        \begin{cases}
			(u_l,c_l)\quad& x<0,\\
			(u_r,c_r)\quad& x>0,	
        \end{cases}
    \end{equation}
  there are four elementary waves on the $(u,c)$ plane as follow:
  \begin{align*}
	&R_1:u-u_l=c-c_l,c>c_l\quad and \quad R_2:u-u_l=-(c-c_l),c<c_l,\\
	&S_1:u-u_l=c-c_l,c<c_l\quad and \quad S_2:u-u_l=-(c-c_l),c>c_l.
  \end{align*}
  On the $(u,c)$ plane, the above waves can divide the upper half-plane $c>0$ into four regions (see \Cref{fg-cqu12}). Moreover, the following result has proved.

  \begin{lemma}\label{lemmaCQU12-1}\cite[Theorem 2.2]{CQu12}
    If 
    \begin{equation}\label{eq:solvability condition}
		u_l-u_r<c_l+c_r,
    \end{equation}
     then the problem \eqref{eq:1-D Euler system}--\eqref{1-D initial condition} admits a self-similar solution. The relation between different state on the $(u,c)$ plane is shown in \Cref{fg-cqu12}. Moreover, the intermediate state is
    \begin{equation}\label{eq:mediate state}
        u_m=\dfrac{u_r+u_l+c_r-c_l}{2},\qquad c_m=\dfrac{c_r+u_r-u_l+c_l}{2}.
    \end{equation}
  \end{lemma}

  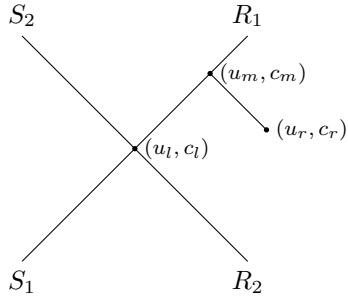
\begin{figure}[H]
	\centering
	\begin{tikzpicture}[smooth, scale=0.5]
        %\draw [step=1,help lines] (0,0) grid (9, 6);
        \draw  (2,0)--(8,6);
        \draw  (2,6)--(8,0);
        \node at (5,3) [right] {\footnotesize$(u_l,c_l)$};
        \node at (7,5) [right] {\footnotesize$(u_m,c_m)$};
        \node at (8.5,3.5) [right] {\footnotesize$(u_r,c_r)$};
        \draw  (7,5)--(8.5,3.5);
        \fill (7,5) circle (2pt);%画点
        \fill (5,3) circle (2pt);%画点
        \fill (8.5,3.5) circle (2pt);%画点
        %\node at (7,3) {I};
        %\node at (5,5) {II};
        %\node at (3,3) {III};
        %\node at (5,1) {IV};
        \node at (2,0) [below] {$S_1$};
        \node at (2,6) [above] {$S_2$};
        \node at (8,0) [below] {$R_2$};
        \node at (8,6) [above] {$R_1$};
	\end{tikzpicture}
	\caption{Four regions associated with a point $(u_l,c_l)$.}
	\label{fg-cqu12}
  \end{figure}

   \begin{remark}\label{re: 1-d initial data}
    It follows from \eqref{initial condition} that the condition \eqref{eq:solvability condition} becomes $u_1-c_1<c_0$. This means that our problem can only discussed under the condition $u_1<c_0+c_1$, which is also valid for the problem considered in \cite{CQ12}.  
   \end{remark}
   
 In two-dimensional case, the interaction of two pressure waves $L_0,L_1$ should be considered. Let us denote by $(u_m,v_m,c_m)$ the state of the flow between $L_0$ and $L_1$; denote by $(u_0,v_0,c_0)$ and $(u_1,v_1,c_1)$ the state on the other side of $L_0$ and $L_1$, respectively. In addition, let $\ell=\frac{c_2}{\tan\alpha/2}$, where $\alpha$ is the angle formed by $L_0$ and $L_1$. Then, we have

  \begin{lemma}\label{lemmaCQU12-2}\cite[Theorem 3.2]{CQu12}
    Assume that $c_i$~$(i=0,1,2)$ and $\ell$ satisfy the conditions
  \begin{align*}
	&c_1+c_0>c_2,\\
	&\arctan{\dfrac{c_0}{\ell}}+\arctan{\dfrac{c_1}{\ell}}-\arctan{\dfrac{c_2}{\ell}}<\frac{\pi}{2},
  \end{align*}
   then the interaction of two different waves can be well defined. In other words, from the interaction point, two waves $\hat{L}_0$ and $\hat{L}_1$ are generated, which are the extension of $L_0$ and $L_1$, respectively. Furthermore,
    \begin{equation*}
        c_m=\ell\tan\Big(\arctan\frac{c_0}{\ell}+\arctan\frac{c_1}{\ell}-\dfrac{\alpha_2}{2}\Big).
    \end{equation*}
  \end{lemma} 

\subsection{Pressure wave structure}\label{sec:2.2}  
  Let us consider the pressure wave structure of the case $u_1<c_1$. Note that the initial boundary value problem \eqref{eq:Euler system} and \eqref{initial condition}--\eqref{slip condition} is invariant under the following scaling
  \begin{equation*}
	(t,x,y)\longmapsto (\zeta t,\zeta x,\zeta y)\qquad \text{for} \quad\zeta\neq 0.
  \end{equation*}
  Then, we can seek a self-similar solution depending on $(\xi,\eta)=({x}/t,{y}/t)$.  

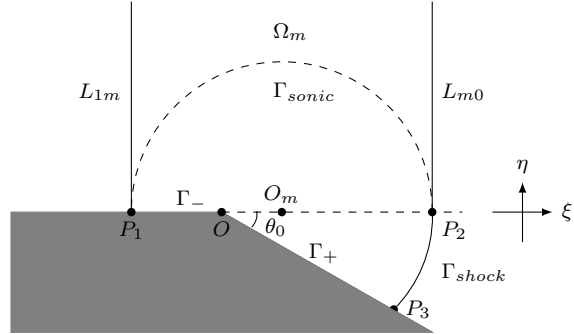
\begin{figure}[H]
	\centering
	\begin{tikzpicture}[smooth, scale=0.8]
        %\draw [step=1,help lines] (0,0) grid (9, 7);
	\draw  [-latex] (9,3)--(10,3) node [right] {\footnotesize$\xi$};
        \draw  [-latex]  (9.5,2.5)--(9.5,3.5)node [above] {\footnotesize$\eta$};
        %\draw  (4.5,3)--(4.5,6);
        \draw  [dashed] (4.5,3)--(8.5,3);
        \draw  [dashed] (8,3) arc (0:180:2.5);
        \draw  (7.36,1.38) arc (315:360:2.2);
        \fill (7.36,1.38) circle (2pt);%画点
        \node at (7.4,1.38) [right] {\footnotesize$P_3$};
	  \draw  (4.5,3)--(8,1);
        \draw [draw=gray, fill=gray, fill opacity=0.6](4.5,3)--(8,1)--(1,1)--(1,3)--(4.5,3);
        %\node at (2.5,4.5){\footnotesize$U_1=(u_1,0,c_1)$};
        %\node at (6.5,4.5){\footnotesize$U_0=(0,0,c_0)$};
        \draw (5,2.75)arc (315:370:0.3);
        \node at (5.4,2.75){\footnotesize$\theta_0$};
        \fill (5.5,3) circle (2pt);%画点
        \node at (5.5,3) [above] {\footnotesize$O_m$};
        \fill (4.5,3) circle (2pt);%画点
        \node at (4.5,3) [below] {\footnotesize$O$};
        \fill (3,3) circle (2pt);%画点
        \node at (3,3) [below] {\footnotesize$P_1$};
        \draw  (3,3)--(3,6.5);
        \node at (3,5) [left] {\footnotesize$L_{1m}$};
        \fill (8,3) circle (2pt);%画点
        \node at (8,3) [below right] {\footnotesize$P_2$};
        \draw  (8,3)--(8,6.5);
        \node at (8,5) [right] {\footnotesize$L_{m0}$};
        \node at (5.2,6) [right] {\footnotesize$\Omega_m$};
        \node at (5.2,5) [right] {\footnotesize$\Gamma_{sonic}$};
        \node at (8,2) [right] {\footnotesize$\Gamma_{shock}$};
        \node at (4,3.2) {\footnotesize$\Gamma_{-}$};
        \node at (6.2,2) [above] {\footnotesize$\Gamma_{+}$};
	\end{tikzpicture}
	\caption{The case $u_1<c_1$.}
	\label{fg-pro2}
  \end{figure}

  Far from the wedge, our problem can be regarded as one-dimensional Riemann problem with initial data \eqref{initial condition} separated by the $y$-axis. The condition \eqref{eq:solvability condition} is always satisfied, provided that $u_1<c_1$ (see \Cref{re: 1-d initial data} above). Hence, two pressure waves $L_{1m}:\xi=u_1-c_1$ and $L_{m0}:\xi=c_0$ are generated. Let $\Omega_m$ be the domain between $L_{1m}$ and $L_{m0}$. By \Cref{lemmaCQU12-1}, the state of the flow $U_m:=(u_m,v_m,c_m)$ in $\Omega_m$ satisfies
    \begin{equation}\label{eq:Um}
		\begin{cases}
		u_m=(u_1-c_1+u_0+c_0)/2=(u_1-c_1+c_0)/2,\\
		v_m=0,\\
            c_m=(c_0+u_0-u_1+c_1)/2=(c_0-u_1+c_1)/2.
		\end{cases}
    \end{equation} 

  We remark here that the two generated waves may be shocks or rarefaction waves. Since the flow is stationary in the domain $-\theta_0\leq\theta<\pi/2$, i.e., $u_0=0$, there are four possible combinations of waves: $1R|2S$ or $1S|2S$ for $u_1>0$, and $1S|2R$ or $1R|2R$ for $u_1<0$, where $S,R$ stand for shocks and rarefaction waves, respectively. Hereafter, without loss of generality, only the case $u_1>0$ will be discussed.
  
  Let $\mathcal{C}_0$ (resp. $\mathcal{C}_1,\mathcal{C}_m$) be the sonic circle with center $(0,0)$ (resp. $(u_1,0), (u_m,v_m)$) and radius $c_0$ (resp. $c_1, c_m$) determined by the state $U_0$ (resp. $U_1,U_m$). Since any pressure wave is characteristic, $L_{1m}$ must be tangent to both $\mathcal{C}_1$ and $\mathcal{C}_m$ at a point $P_1(u_1-c_1,0)$, and $L_{m0}$ must be tangent to both $\mathcal{C}_m$ and $\mathcal{C}_0$ at a point $P_2(c_0,0)$. Note that the point $P_1$ lies on the negative $\xi$-axis owing to $u_1<c_1$, i.e., the origin $O$ lies in the interior of $\mathcal{C}_m$. Due to the impact of the boundary, $L_{m0}$ should be bent from the point $P_2$ to ensure that it intersects the boundary $\partial W_+$ perpendicularly at a point $P_3$. Obviously, the arc ${P_2P_3}$ is a part of the circle $\mathcal{C}_0$, and the coordinates of $P_3$ is $(c_0\cos\theta_0,-c_0\sin\theta_0)$. Then, we can draw the structure of pressure waves as shown in \Cref{fg-pro2}. 

\subsection{Main theorem}\label{sec:2.3}   
  As mentioned above, the flow under consideration is potential flow. Then, let us introduce a velocity potential $\Phi$ by $(u,v)=(\partial_x \Phi,\partial_y \Phi)$. Correspondingly, we seek a solution with the form $\rho(t,x,y)=\rho(\xi,\eta), \Phi(t,x,y)=t\phi(\xi,\eta)$ in the self-similar coordinates $(\xi,\eta)$. By \eqref{eq:Euler system}, a trivial calculation yields that the function $\phi$ satisfies
      \begin{equation}\label{eq:phi}
     \mathrm{div}(\rho D\phi)-(\xi\rho_{\xi}+\eta\rho_\eta)=0,
      \end{equation}
  or
    \begin{equation}\label{eq:phi second order}
     \big(c^2-(\phi_\xi-\xi)^2\big)\phi_{\xi\xi}-2(\phi_\xi-\xi)(\phi_\eta-\eta)\phi_{\xi\eta}+\big(c^2-(\phi_\eta-\eta)^2\big)\phi_{\eta\eta}=0,
      \end{equation}
  where $\mathrm{div}$ and $D$ stand for the divergent and the gradient operator with respect to $(\xi,\eta)$, respectively; $\rho$ or $c$ is given by the following Bernoulli equation
      \begin{equation}\label{eq:Bernoulli for phi}
     \frac{1}{2}|D\phi|^2+\phi-(\xi\phi_{\xi}+\eta\phi_\eta)-\frac{A}{2\rho^2}=B_0.
      \end{equation}
  Here, the constant $A$ is given in \eqref{eq:Chaplygin gas}; $B_0$ is the Bernoulli constant determined by the initial data. Also, let $\varphi=\phi-\frac{1}{2}(\xi^2+\eta^2)$ be the pseudovelocity potential. Then equation \eqref{eq:phi} becomes
      \begin{equation}\label{eq:varphi}
     \mathrm{div}(\rho D\varphi)+2\rho=0.
      \end{equation}
      
  Before proceeding, let us give a subsonic domain $\Omega_{sub}$. We denote by $\Gamma_{sonic}$ and $\Gamma_{shock}$ the arcs ${P_1P_2}$ and ${P_2P_3}$, respectively; denote by $\Gamma_{+}$ and $\Gamma_{-}$ the lines $OP_3$ and $OP_1$, respectively. Then the domain $\Omega_{sub}$ is bounded by $\Gamma_{sonic}$, $\Gamma_{shock}$, $\Gamma_{+}$ and $\Gamma_{-}$. The state of the flow is pseudosupersonic and known in $\Omega\setminus\Omega_{sub}$, but to be determined in $\Omega_{sub}$. It is easy to check that equation \eqref{eq:varphi} (or \eqref{eq:phi}) is hyperbolic in $\Omega\setminus\overline{\Omega_{sub}}$, and parabolic degenerate on the boundary $\overline{\Gamma_{shock}}\cup\overline{\Gamma_{sonic}}$. Then, the stata of the flow in $\Omega_{sub}$ is expected to be pseudosubsonic; that is, \eqref{eq:varphi} (or \eqref{eq:phi}) is elliptic in $\Omega_{sub}$. Furthermore, it is shown in \cite{CQ12,Serre09} that \eqref{eq:varphi} is uniformly elliptic in the region that is strictly away from $\overline{\Gamma_{shock}}\cup\overline{\Gamma_{sonic}}$. Therefore, it is natural to seek a solution so that \eqref{eq:varphi} (or \eqref{eq:phi}) is also locally uniform elliptic in $\Omega_{sub}$.
  
  Then, Problem \ref{prob1} can be reformulated for the function $\phi$ as
  
  \begin{problem}\label{prob2}
	We wish seek a solution 
 $$\phi \in C^{0,1}(\overline{\Omega_{sub}})\cap C^3(\overline{\Omega_{sub}}\setminus(\{O\}\cup\overline{\Gamma_{sonic}}\cup\overline{\Gamma_{shock}}))$$
 satisfying equation \eqref{eq:phi} in $\Omega_{sub}$ with the boundary conditions:
 \begin{align}
     \phi=\phi_0=0\quad&\text{on}\quad {\Gamma_{shock}},\label{b shock}\\
     \phi=\phi_m=u_m\xi\quad&\text{on}\quad {\Gamma_{sonic}},\label{b sonic}\\
    D\phi\cdot \bm{\nu}=0\quad&\text{on}\quad \Gamma_{+}\cup\Gamma_{-}.\label{b Neu boundary}
 \end{align}
  Moreover, equation \eqref{eq:phi} is elliptic in $\Omega_{sub}$ and uniformly elliptic in the domain away from the degenerate boundary $\overline{\Gamma_{shock}}\cup\overline{\Gamma_{sonic}}$.
 \end{problem}

 For the regularity of $\phi$ across the sonic boundary $\Gamma_{sonic}$, the flow is continuous across the sonic line for the polytropic gas, as given by Chen and Feldman in the problem of shock reflection \cite{CF10} (also see \cite{Bae09}). However, to our best knowledge, there are no relevant conclusions for the Chaplygin gas. In this paper, we will show that if the flow is continuous across the sonic boundary $\Gamma_{sonic}$, then the solution given in Problem \ref{prob2} does not exist for the case $u_1<c_1$, but exists for the case $c_1<u_1<c_0+c_1$. 

Now we can present our main theorem as

\begin{theorem}\label{thm: Main2}
For \Cref{prob1}, we have the following conclusion:

$(i)$ For the case $u_1<c_1$, \Cref{prob1} does not permit a self-similar solution as in \Cref{prob2} under the assumption that the flow is continuous across the sonic boundary $\Gamma_{sonic}$.

$(ii)$ For the case $c_1<u_1<c_0+c_1$, assume that $\theta_0$ is small, then \Cref{prob1} exists a piecewise smooth solution.
\end{theorem}

\begin{remark}
From the discussion in \Cref{sec:2.2}, we can see that in the case $u_1<c_1$, there is no wave interaction. However, the interaction of different waves will occur in the case $c_1<u_1<c_0+c_1$ (see \Cref{sec:4} below). To ensure that the interaction of two waves is well-defined, we should assume that $\theta_0$ is small by \Cref{lemmaCQU12-2}. 
\end{remark}

For the case $u_1<c_1$, inspired by the work in \cite{X20}, we first analyze the sign of the velocity along the $\xi$-axis $u=\phi_\xi$ in $\Omega_{sub}$, and then obtain our conclusion by deriving a contradiction. For the case $c_1<u_1<c_0+c_1$, we can establish the pressure wave structure, and then reduce our problem to a Dirichlet problem \eqref{for u1>c1 dir} below, which can be solved by the result in \cite{Serre09}.

\section{The case $u_1<c_1$}\label{sec:3} 
\subsection{Some useful estimates}\label{sec:3.1} 
Let us first list some useful estimates for the solution $\phi$. Let $r_0:=\frac{1}{2}\text{dist}(O,\Gamma_{shock}\cup\Gamma_{sonic})$, and let $C_r(O)$ be the circle with center the origin $O$ and radius $r$.

  \begin{lemma}\label{lemmaX20-1}
    Let $\phi$ be a solution of Problem \ref{prob2}. For any $r\in(0,r_0)$, if 
    \begin{equation*}
       \|D\phi\|_{C(C_r(O)\cap\Omega_{sub})}\leq \omega(r),
    \end{equation*}
   then, 
    \begin{equation*}
       |D^2\phi|\leq C\dfrac{\omega(2r)}{r}\quad\text{on}\quad \partial C_r(O)\cap\Omega_{sub},
    \end{equation*}
    where $\omega(r)$ is a nondecreasing continuous function on $[0,r_0]$; the constant $C$ depends on the elliptic ratio $\lambda$ and $\|D\phi\|_{C(C_{r_0}(O)\cap\Omega_{sub})}$.
  \end{lemma}

  \begin{lemma}\label{lemmaX20-2}
    If $\phi$ is a solution of Problem \ref{prob2}, then we have
    \begin{equation*}
      \phi\in C^1(\overline{C_{r_0}(O)\cap\Omega_{sub}}). 
    \end{equation*}
   Moreover,
    \begin{equation*}
      D\phi=\bm{0}\quad\text{at the origin}\quad O. 
    \end{equation*}
  \end{lemma}

  Since \eqref{eq:phi} is uniformly elliptic in the domain away from the degenerate boundary $\overline{\Gamma_{shock}}\cup\overline{\Gamma_{sonic}}$, \Cref{lemmaX20-1,lemmaX20-2} can be proved directly by applying the discussion in \cite[Lemmas 3.1--3.3]{X20}. 
\subsection{The sign of $u$}\label{sec: for u} 
 Now we turn to prove that if $\phi$ is a solution of Problem \ref{prob2} and the flow is continuous across the sonic boundary $\Gamma_{sonic}$, then the velocity along the $\xi$-axis $u=\phi_\xi>0$ in the interior of $\Omega_{sub}\cup\Gamma_{+}\cup\Gamma_{-}$. To this purpose, we first prove the following lemma.

   \begin{lemma}\label{for velocity u}
    Assume that $\phi\in C^{0,1}(\overline{\Omega_{sub}})\cap C^3(\overline{\Omega_{sub}}\setminus(O\cup\overline{\Gamma_{sonic}}\cup\overline{\Gamma_{shock}}))$ satisfies the elliptic equation
     \begin{equation}\label{linear phi}     a_{11}\phi_{\xi\xi}+2a_{12}\phi_{\xi\eta}+a_{22}\phi_{\eta\eta}=0\quad\text{in}\quad \Omega_{sub}, 
     \end{equation} 
     and the boundary condition
     \begin{equation}\label{linear Neu}
        D\phi\cdot \bm{\nu}=0\quad\text{on}\quad \Gamma_{+}\cup\Gamma_{-}, 
     \end{equation}
    where the coefficient $a_{ij}\in C^1(\overline{\Omega_{sub}}\setminus(O\cup\overline{\Gamma_{sonic}}\cup\overline{\Gamma_{shock}})$, $i,j=1,2$. If \eqref{linear phi} is uniformly elliptic in the domain away from $\overline{\Gamma_{sonic}}\cup\overline{\Gamma_{shock}}$, then $u=\phi_\xi$ cannot achieve the local minimum anywhere in the interior of $\Omega_{sub}\cup\Gamma_{+}\cup\Gamma_{-}$ unless it is constant.
  \end{lemma}

  \begin{proof}
     Let us first obtain the equation of $u=\phi_{\xi}$ by taking $\partial_\xi$ of \eqref{linear phi}
     
    \begin{equation}\label{linear u}
     \begin{split}
        0&=(a_{11}\phi_{\xi\xi}+2a_{12}\phi_{\xi\eta}+a_{22}\phi_{\eta\eta})_\xi\\      &=a_{11}u_{\xi\xi}+2a_{12}u_{\xi\eta}+a_{22}u_{\eta\eta}+(a_{11})_\xi \phi_{\xi\xi}+2(a_{12})_\xi\phi_{\xi\eta}+(a_{22})_\xi\phi_{\eta\eta}\\        &=a_{11}u_{\xi\xi}+2a_{12}u_{\xi\eta}+a_{22}u_{\eta\eta}+(a_{11})_\xi u_{\xi}+2(a_{12})_\xi u_{\eta}-\frac{(a_{22})_\xi}{a_{22}}(a_{11}u_{\xi}+2a_{12}u_{\eta}),
     \end{split}
    \end{equation}
    where we have employed equation \eqref{linear phi} to derive the last equality.
    From the condition of this lemma, it follows that equation \eqref{linear u} is locally uniformly elliptic. Thus, by the strong maximum principle, $u$ cannot achieve the local minimum anywhere in $\Omega_{sub}$ unless it is constant.

    Next, we consider the boundary $\Gamma_{+}\cup\Gamma_{-}$. From the boundary condition $D\phi\cdot\bm{\nu}=\phi_\eta=0$ on $\Gamma_{-}$, we have
    \begin{equation}\label{Neu condition for u1}
        u_\eta=0 \quad\text{on}\quad \Gamma_{-},
    \end{equation}
   by taking the tangential derivative along $\Gamma_{-}$. 

   For the boundary $\Gamma_{+}$, we can make a rotation of coordinates as follows
 \begin{equation}\label{rotation}
		\begin{cases}			\tilde{\xi}=\xi\cos\theta_0-\eta\sin\theta_0,\\	\tilde{\eta}=\xi\sin\theta_0+\eta\cos\theta_0.\\  \end{cases}
    \end{equation}
  It is obvious that equation \eqref{linear phi} has the same form, and is also locally uniformly elliptic in the new coordinates $(\tilde{\xi},\tilde{\eta})$. Moreover, the boundary $\Gamma_{+}$ lies on the positive $\tilde{\xi}$--axis in the coordinates $(\tilde{\xi},\tilde{\eta})$. Hence, we get
    \begin{equation}\label{Neu condition for u2}
        u_{\tilde{\eta}}=0 \quad\text{on}\quad \Gamma_{+}.
    \end{equation}
  A direct verification shows that the boundary conditions \eqref{Neu condition for u1} and \eqref{Neu condition for u2} are oblique. Then, this completes the proof by applying the Hopf lemma.  
  \end{proof} 

 By using Lemma \ref{for velocity u}, we can obtain the following result.

 \begin{proposition}\label{key point}
   Let $\phi$ be a solution of Problem \ref{prob2}. Assume that the flow is continuous across the sonic boundary $\Gamma_{sonic}$. Then
     \begin{equation}\label{u>0}
        u=\phi_\xi>0 \quad\text{in} \quad\Omega_{sub}\cup\Gamma_{+}\cup\Gamma_{-}.
     \end{equation}      
  \end{proposition}

  \begin{proof}
      To prove this proposition, we only need to verify that $u=\phi_\xi\geq 0$ on $O\cup\overline{\Gamma_{sonic}}\cup\overline{\Gamma_{shock}}$ by \Cref{for velocity u}. It follows from \Cref{lemmaX20-2} that $u=\phi_\xi = 0$ at the origin $O$. Moreover, under the assumption that the flow is continuous across the sonic boundary, we have
    \begin{equation}
        u=\phi_\xi=(u_m,0)\cdot(1,0)=u_m>0 \quad\text{on} \quad\Gamma_{sonic}.
     \end{equation}       
     Here, we recall that $u_m>0$ because $u_1>0$ mentioned in \Cref{sec:2.2}.

     Next, we focus on the boundary $\overline{\Gamma_{shock}}$ and the point $P_1$ in \Cref{fg-pro2}. Notice that \eqref{eq:phi} (or \eqref{eq:phi second order}) is uniformly elliptic in the domain away from the boundary $\overline{\Gamma_{sonic}}\cup\overline{\Gamma_{shock}}$, and then the Hopf lemma can be applied on the boundary $\Gamma_{+}\cup\Gamma_{-}$ and the origin $O$. Using the Hopf lemma and the maxiumum principle, we can see that 
    \begin{equation}\label{max for phi}
      \underset{\small{\Omega_{sub}}}{\large\sup} \phi=\underset{\small{\overline{\Gamma_{sonic}}\cup\overline{\Gamma_{shock}}}}{\large\sup} \phi.
     \end{equation} 
     On the other hand, \eqref{eq:phi} is degenerate on the boundary $\overline{\Gamma_{sonic}}\cup\overline{\Gamma_{shock}}$,     which means that $(\phi_\xi-\xi)^2+(\phi_\eta-\eta)^2=c^2$. Combing this fact and the Bernoulli equation \eqref{eq:Bernoulli for phi}, the boundary conditions \eqref{b shock}--\eqref{b sonic} are equivalent to     
     \begin{equation}\label{boundary for phi2}
     \phi=\frac{1}{2}(\xi^2+\eta^2)+B_0\quad\text{on}\quad \overline{\Gamma_{sonic}}\cup\overline{\Gamma_{shock}}.
     \end{equation}
      With the use of \eqref{max for phi}--\eqref{boundary for phi2} and the pressure wave structure shown in \Cref{fg-pro2}, we know that the maximum value of $\phi$ is attained at every point of $\overline{\Gamma_{shock}}$, and the minimum value of $\phi$ is only attained at the point $P_1$, which implies
    \begin{equation}
        u=\phi_\xi\geq 0 \quad\text{on} \quad\overline{\Gamma_{shock}}\cup P_1.
     \end{equation}  

     In conclusion, we complete the proof of the lemma.
  \end{proof}

\subsection{Proof of the case $(i)$ of \Cref{thm: Main2}}\label{sec:3.2} 
 We are going to prove the case $(i)$ of \Cref{thm: Main2} by using \Cref{key point}.

 It follows from \eqref{u>0} that for any $r\in (0,r_0)$, there must exist a number $\delta>0$, such that
    \begin{equation}\label{assume for delta}
        u=\phi_\xi\geq \delta>0 \quad\text{on} \quad\partial C_r(O)\cap \Omega_{sub}.
    \end{equation} 
  Define a function
    \begin{equation}\label{for h}
		f:=(u-\delta)^-=
       \begin{cases}			
        0,&\quad \text{if} \quad u\geq\delta,\\
        u-\delta, &\quad \text{if}  \quad u<\delta.
        \end{cases}
    \end{equation}
  Moreover, for any $\varepsilon\in (0,\frac{r}{2})$, we define a domain
     \begin{equation}\label{for D}
        D_{r\setminus\varepsilon}:=\big(C_r(O)\setminus C_\varepsilon(O)\big)\cap \Omega_{sub},
     \end{equation} 
 with four vertices $E_1:=\partial C_r(O)\cap\Gamma_{-}, E_2:=\partial C_r(O)\cap\Gamma_{+}, F_1:=\partial C_\varepsilon(O)\cap\Gamma_{-}$ and $F_2:=\partial C_\varepsilon(O)\cap\Gamma_{+}$.

  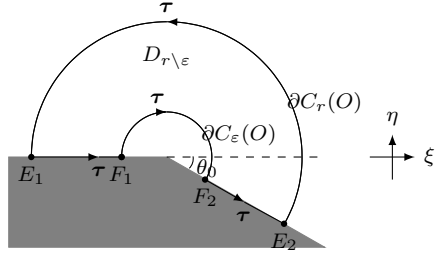
\begin{figure}[H]
	\centering
	\begin{tikzpicture}[smooth, scale=0.6]
       %\draw [step=1,help lines] (0,0) grid (9, 7);
	\draw  [-latex] (9,3)--(10,3) node [right] {\footnotesize$\xi$};
        \draw  [-latex]  (9.5,2.5)--(9.5,3.5)node [above] {\footnotesize$\eta$};
        \draw [draw=gray, fill=gray, fill opacity=0.6](4.5,3)--(8,1)--(1,1)--(1,3)--(4.5,3); 
        \draw [dashed] (4.5,3)--(8,3);
        \draw (5,2.75)arc (315:370:0.3);
        \node at (5.4,2.75){\footnotesize$\theta_0$};
        \draw [domain=-30:180] plot ({4.5+3*cos(\x)}, {3+3*sin(\x)});
        \fill (3.5,3) circle (2pt);%画点
        \draw [-latex, domain=-30:90] plot ({4.5+3*cos(\x)}, {3+3*sin(\x)});
        \draw [domain=90:180] plot ({4.5+3*cos(\x)}, {3+3*sin(\x)});
        \fill (3.5,3) circle (2pt);%画点
        \node at (3.5,3) [below] {\footnotesize$F_1$};
        \fill (1.5,3) circle (2pt);%画点
        \node at (1.5,3) [below] {\footnotesize$E_1$};
        \draw [domain=-30:180] plot ({4.5+1*cos(\x)}, {3+1*sin(\x)});
        \draw [-latex,domain=180:90] plot ({4.5+1*cos(\x)}, {3+1*sin(\x)});
        \draw [domain=90:-30] plot ({4.5+1*cos(\x)}, {3+1*sin(\x)});
        \node at (4.5,6)[above] {\footnotesize$\bm{\tau}$};
        \node at (4.3,4)[above] {\footnotesize$\bm{\tau}$};
         \node at (3,3)[below] {\footnotesize$\bm{\tau}$};
          \node at (6.2,2)[below] {\footnotesize$\bm{\tau}$};
        \fill (5.35,2.5) circle (2pt);%画点      
        \node at (5.35,2.5) [below] {\footnotesize$F_2$};
        \fill (7.1,1.52) circle (2pt);%画点
        \node at (7.1,1.52) [below] {\footnotesize$E_2$};
        \node at (4.5,5.2) {\footnotesize$D_{r\setminus\varepsilon}$};
        \node at (6.1,3.5) {\footnotesize$\partial C_\varepsilon(O)$};
        \node at (8,4.2) {\footnotesize$\partial C_r(O)$};
        \draw[-latex] (1.5,3) -- (3,3); %画箭头
        \draw (3,3) -- (3.5,3); %画箭头
        \draw[-latex] (5.35,2.5) -- (6.22,2); %画箭头
        \draw(6.22,2) -- (7.1,1.52); %画箭头
	\end{tikzpicture}
	\caption{The case $u_1<c_1$.}
	\label{fg-pro5}
  \end{figure}

 For convenience, equation \eqref{eq:phi second order} can be written as
    \begin{equation}\label{phi2}  \widetilde{a_{11}}\phi_{\xi\xi}+2\widetilde{a_{12}}\phi_{\xi\eta}+\widetilde{a_{22}}\phi_{\eta\eta}=0, 
    \end{equation} 
 where $\widetilde{a_{ij}}$ has its expression given in \eqref{eq:phi second order}. Thanks to the ellipticity of equation \eqref{eq:phi second order}, there is $\widetilde{a_{22}}>0$ in $\Omega_{sub}\cup\Gamma_{+}\cup\Gamma_{-}$, and the following integral is always nonnegative
    \begin{equation}\label{I}  
    \textbf{N}:=\int_{D_{r\setminus\varepsilon}\cap\{u<\delta\}} 
    \Big(\frac{\widetilde{a_{11}}}{\widetilde{a_{22}}}u^2_\xi+\frac{2\widetilde{a_{12}}}{\widetilde{a_{22}}}u_\xi u_\eta+u^2_{\eta}\Big) \hspace{1mm}d\xi d\eta\geq 0.
    \end{equation} 
 We will arrive at a contradiction as $\varepsilon\rightarrow 0$.

 Before proceeding, we recall some useful equalities here. Owing to $\phi\in C^3(\overline{\Omega_{sub}}\setminus(O\cup\overline{\Gamma_{sonic}}\cup\overline{\Gamma_{shock}})$ and $(u,v)= (\phi_\xi, \phi_\eta)$, then 
    \begin{equation}\label{eq:u=v}  
    u_\eta=v_\xi\quad \text{in}\quad \overline{C_r(O)}\setminus O.
    \end{equation} 
 From \eqref{for h}, and using \eqref{assume for delta}, we find 
    \begin{equation}\label{eq:h}    
   f=0 \quad \text{on}\quad \partial C_r(O)\cap \Omega_{sub}.
    \end{equation} 
 Besides, equation \eqref{phi2} implies that 
    \begin{equation}\label{eq:for v}  
  v_\eta=-\Big(\frac{\widetilde{a_{11}}}{\widetilde{a_{22}}}u_\xi+\frac{2\widetilde{a_{12}}}{\widetilde{a_{22}}}u_\eta\Big).
    \end{equation}  

 With use of the relations \eqref{eq:u=v}--\eqref{eq:for v}, a trivial calculation yields
 \begin{equation}
     \begin{split}
         \textbf{N}&=\int_{D_{r\setminus\varepsilon}} 
    \Big(\frac{\widetilde{a_{11}}}{\widetilde{a_{22}}}u_\xi f_\xi+\frac{2\widetilde{a_{12}}}{\widetilde{a_{22}}} u_\eta f_\xi+u_{\eta}f_\eta\Big)\hspace{1mm} d\xi d\eta\\      
    &=\int_{D_{r\setminus\varepsilon}} 
    -v_\eta f_\xi+v_\xi f_\eta \hspace{1mm} d\xi d\eta =\int_{D_{r\setminus\varepsilon}} 
    -(v_\eta f_\xi+v_{\xi\eta}f)+(v_{\xi\eta}f+v_\xi f_\eta) \hspace{1mm} d\xi d\eta\\
    &=-\int_{D_{r\setminus\varepsilon}} 
    D\cdot(v_\eta f,-v_\xi f) \hspace{1mm} d\xi d\eta=-\oint_{\partial D_{r\setminus\varepsilon}} 
   v_{\bm{\tau}} f \hspace{1mm} ds\\
    &=-\int_{\partial C_\varepsilon(O)\cap\Omega_{sub}} 
   v_{\bm{\tau}} f \hspace{1mm} ds-\int_{\overline{E_1F_1}} 
   v_{\bm{\tau}} f \hspace{1mm} ds-\int_{\overline{F_2E_2}} 
   v_{\bm{\tau}} f \hspace{1mm} ds,\\
     \end{split}
    \end{equation}
   where $\bm{\tau}$ is the unit tangential vector of $\partial D_{r\setminus\varepsilon}$, and its direction is given in Figure \ref{fg-pro5}; the line integral is respect to the arc length $s$. 

   From the boundary condition \eqref{b Neu boundary}, it follows that
    \begin{align*}
     v=0 ,\quad& \text{on} \quad\overline{E_1F_1},\\
     v=-u\tan\theta_0 ,\quad& \text{on} \quad\overline{F_2E_2},
    \end{align*}
   which indicates
   \begin{equation}\label{for E1F1}
     \int_{\overline{E_1F_1}}v_{\bm{\tau}} f \hspace{1mm} ds=0
   \end{equation}
   and
    \begin{equation}
    \begin{split}
      \int_{\overline{F_2E_2}} 
   v_{\bm{\tau}} f \hspace{1mm} ds&= \int_{\overline{F_2E_2}} -u_{\bm{\tau}} f\tan\theta_0 \hspace{1mm} ds\\
   &= \int_{\overline{F_2E_2}} -f_{\bm{\tau}} f\tan\theta_0 \hspace{1mm} ds\\
   &= -\int_{\overline{F_2E_2}}\frac{1}{2}(f^2)_{\bm{\tau}} \tan\theta_0 \hspace{1mm} ds\\
   &= \frac{1}{2}(f^2(F_2)-f^2(E_2))\tan\theta_0=\frac{1}{2}f^2(F_2)\tan\theta_0,
     \end{split}
     \end{equation}
   where $f(E_2)=0$ due to the point $E_2$ on the boundary $\partial C_r(O)\cap \Omega_{sub}$. Notice that $D\phi=0$ at the origin $O$  by Lemma \ref{lemmaX20-2}. Thus, $f(F_2)$ converges to $-\delta$ as $\varepsilon\rightarrow 0$, and
   \begin{equation}\label{for F2E2}
     \int_{\overline{F_2E_2}} 
   v_{\bm{\tau}} f \hspace{1mm} ds\xrightarrow{\varepsilon\rightarrow 0}\frac{1}{2}\delta^2\tan\theta_0.
   \end{equation}

   According to \Cref{lemmaX20-2}, there must exist a nondecreasing continuous function $g(\varepsilon)$ on $[0,r_0)$, so that for any $\varepsilon\in (0,r_0)$,
    \begin{equation}\label{eq f(r)}
      \|D\phi\|_{C(C_\varepsilon(O)\cap\Omega_{sub})}\leq g(\varepsilon),
    \end{equation}
    and $g(\varepsilon)\rightarrow 0$ as $\varepsilon\rightarrow 0$. Then, using \eqref{eq f(r)} and \Cref{lemmaX20-1}, we known that
    \begin{equation}\label{condition 1}
      |D^2\phi|\leq C \dfrac{g(2\varepsilon)}{\varepsilon},\quad\text{on}\quad \partial C_\varepsilon(O)\cap\Omega_{sub},
    \end{equation}
   where $C$ is a constant depending on $\|D\phi\|_{C(C_{r_0}(O)\cap\Omega_{sub})}$. Besides, by \eqref{for h}, there is
    \begin{equation}\label{condition 2}
      |f|\leq ||D\phi||_{L^\infty(\Omega_{sub})}+\delta,\quad\text{in}\quad  \Omega_{sub}.
    \end{equation}
   Therefore, applying \eqref{condition 1}-\eqref{condition 2}, we obtain
   \begin{equation}\label{for circle arc}
    \int_{\partial C_\varepsilon(O)\cap\Omega_{sub}} 
   v_{\bm{\tau}} f \hspace{1mm} ds\leq 2\pi\varepsilon C \dfrac{g(2\varepsilon)}{\varepsilon}
\xrightarrow{\varepsilon\rightarrow 0}0.
   \end{equation}
%\leq \big(||D\phi||_{L^\infty(\Omega_{sub})}+\delta\big)\int_{\partial C_\varepsilon(O)\cap\Omega_{sub}} |D^2\phi| \hspace{1mm} ds 

  Using the above estimates \eqref{for E1F1}, \eqref{for F2E2} and \eqref{for circle arc}, we conclude that
   \begin{equation*}
   \textbf{N}\xrightarrow{\varepsilon\rightarrow 0}-\frac{1}{2}\delta^2\tan\theta_0<0,
   \end{equation*}
   which is a contradiction to \eqref{I}. Hence, the case $(i)$ of \Cref{thm: Main2} has proved.
   
\section{The case $c_1<u_1<c_0+c_1$}\label{sec:4}  
  In this section, we study the case $c_1<u_1<c_0+c_1$. Let us still use the notations defined in Section \ref{sec:2}. Obviously, there are two possible combinations for $L_{1m}:\xi=u_1-c_1$ and $L_{m0}:\xi=c_0$: $1R|2S$ or $1S|2S$. Besides, we point out here that the sonic circle $\mathcal{C}_0$ can always intersect with the boundary $\Gamma_{+}$ because $\theta_0$ is small.

  The analysis of the pressure wave structure is similar to the case $u_1>c_1$ in \cite[p. 680-683]{CQ12}. To avoid repetition, we only give the necessary instructions.

 \begin{figure}[H]
	\centering
	\begin{tikzpicture}[smooth, scale=0.5]
        %\draw [step=1,help lines] (0,0) grid (14, 7);
	\draw  [-latex] (11,3)--(12,3) node [right] {\footnotesize$\xi$};
        \draw  [-latex]  (11.5,2.5)--(11.5,3.5)node [above] {\footnotesize$\eta$};
        %\draw  (4.5,3)--(4.5,6);
        \draw  [dashed] (2.2,3)--(9.5,3);
        \draw [draw=gray, fill=gray, fill opacity=0.6](2.2,3)--(10,0.5)--(1.8,0.5)--(1.8,3)--(2.5,3);
       \draw [thick,domain=0:115] plot ({6.25+2.75*cos(\x)}, {3+2.75*sin(\x)});
       \draw [dashed,domain=115:180] plot ({6.25+2.75*cos(\x)}, {3+2.75*sin(\x)});
       \fill (6.25,3) circle (2pt);%画点
       \node at (6.25,3) [below] {\footnotesize$O_m$};
       \draw [thick,domain=0:-21.4] plot ({3.5+5.5*cos(\x)}, {3+5.5*sin(\x)});
        \fill (9,3) circle (2pt);%画点
        \node at (9,3) [below right] {\footnotesize$P_2$};
        \fill  (8.61,1.0) circle (2pt);%画点
        \node at (8.61,1.1) [right] {\footnotesize$P_3$};
        \draw  (9,3)--(9,8);
        \node at (9,5) [right] {\footnotesize$L_{m0}$};
        \fill (2.2,3) circle (2pt);%画点
        \node at (2.2,3) [below] {\footnotesize$O$};
        \draw  (3.5,4.7)--(3.5,8);
        \node at (3.5,6) [left] {\footnotesize$L_{1m}$};
        \fill (3.5,3) circle (2pt);%画点
        \node at (3.7,3.15) [below] {\footnotesize$P_0$};
        \draw  [dashed] (3.5,4.7)--(3.5,3);
        \draw  (3.5,4.7)--(2.2,3);
        \fill (3.5,4.7) circle (2pt);%画点
        \node at (3.5,4.7) [left] {\footnotesize$P$};
        \draw  (3.5,4.7)--(5,5.45);
        \draw  [dashed] (6.25,3)--(5,5.45);
        \fill (5,5.45) circle (2pt);%画点
        \node at (5,5.35) [right] {\footnotesize$T_m$};
        \fill (4.42,5.88) circle (2pt);%画点
        \node at (4.42,5.88) [left] {\footnotesize$T_1$};
        \fill (13,-0.4) circle (2pt);%画点
        \node at (13,-0.4) [below] {\footnotesize$O_2$};
        \draw  [dashed] (13,-0.4)--(8.4,3);
        \draw  (13,-0.4)--(10,0.5);
        \draw  [dashed] (3.5,4.7)--(4.42,5.92);
        \draw  [dashed] (8.4,3)--(5,5.45)--(4.42,5.92);
        \draw [dashed,domain=110:180] plot ({8.4+4.9*cos(\x)}, {3+4.9*sin(\x)});
        \fill (8.4,3) circle (2pt);%画点
        \node at (8.4,3) [below] {\footnotesize$O_1$};
       %\draw [thick,domain=110:170] plot ({5.75+2.5*cos(\x)}, {3+2.5*sin(\x)});
       \draw (5,5.45)arc (110:170:2.82);
       \fill (6.8,1.9) circle (2pt);%画点
       \draw  [dashed] (5,5.45)--(6.8,1.9);
       \node at (6.8,1.9) [below] {\footnotesize$O_3$};
       \draw  (3.5,4.7)--(3,2.7);
       \fill (3,2.73) circle (2pt);%画点
       \node at (3,2.73) [below] {\footnotesize$P_1$};
       \fill (3.16,3.27) circle (2pt);%画点
       \node at (3.3,3.3) [left] {\footnotesize$T_2$};
       \draw  [dashed] (3.16,3.27)--(13,-0.4);
	\end{tikzpicture}
	\caption{The case $c_1<u_1<c_0+c_1$.}
	\label{fg-pro3}
  \end{figure}
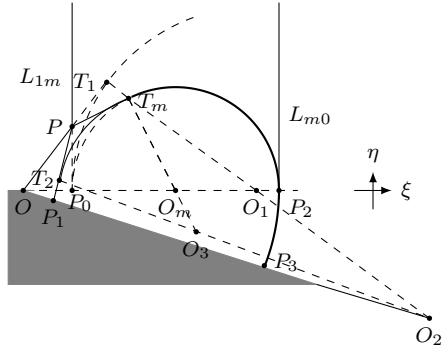

  Unlike the case $u_1<c_1$, a pressure wave, denoted by $L_{12}$ (i.e. ${OP}$ in \Cref{fg-pro3}), arises from the origin $O$ because the flow at the origin is surpersonic. Then, the interaction of $L_{1m}$ and $L_{12}$ occurs at a point $P$. From \Cref{lemmaCQU12-2}, it follows that the interaction of $L_{1m}$ and $L_{12}$ can be well defined, provided that $\theta_0$ is small. Since any pressure wave is a characteristic, the location of the resulting wave can be determined, i.e., ${PT_m}$ and ${PT_2}$ in \Cref{fg-pro3}. For reader’s convenience, we give the expression for the waves in \Cref{fg-pro3}, and omit the detailed calculation.
  \begin{alignat}{2}
	{OP}&:\quad \eta&&=\xi\tan\beta,\label{eq1:1.7}\\
	{PT_m}&:\quad
	\eta&&=(\xi-\xi_{P})\tan(\beta_m-\dfrac{\pi}{2})+\eta_{P},\label{eq1:1.8}\\
	{PT_2}&:\quad
	 \eta&&=(\xi-\xi_{P})\tan(\frac{\pi}{2}-(\beta_3-\beta_m))+\eta_{P}.\label{eq1:1.9}
  \end{alignat}
  Here $\beta=\arcsin{(c_1/u_1)}>0$; $(\xi_{P},\eta_{P})=(u_1-c_1,(u_1-c_1)\tan\beta)$ is the coordinate of the point $P$; $\beta_m=2\arctan{(c_m/\eta_{P})}$ is the angle formed by ${PT_m}$ and ${PP_0}$; $\beta_3=\angle T_1PT_2+\angle T_mPP_0-\angle T_1PP_0$. Denote by $U_2(u_2,v_2,c_2)$ and $U_3(u_3,v_3,c_3)$ the uniform states of the flow in the supersonic regions $OPP_1$ and $PP_1T_m$, respectively. Then, the sonic circles $\mathcal{C}_2$ and $\mathcal{C}_3$ can be determined by $U_2$ and $U_3$, respectively, in the same form as $\mathcal{C}_0$ and $\mathcal{C}_m$. The arcs $P_2P_3, P_2T_m, P_1T_2, T_2T_m$ are the part of the circles $\mathcal{C}_0, \mathcal{C}_m, \mathcal{C}_2, \mathcal{C}_3$, respectively. Moreover, the states $U_2$ and $U_3$ are
  \begin{flalign}\label{eq:U2}
  \left\{{
  \begin{array}{ll}
      u_2=u_1+(c_2-c_1)\sin\beta,\\
      v_2=(c_1-c_2)\cos\beta,\\
      c_2=c_1\cot\beta\tan(\beta+\theta_0).
  \end{array}}
  \right.,
  \left\{{
    \begin{array}{ll}
       u_3=\xi_{P}-c_3\cos(\frac{\pi}{2}+\beta_m-\frac{\beta_3}{2})\csc\frac{\beta_3}{2},\\
       v_3=\eta_{P}-c_3\sin(\frac{\pi}{2}+\beta_m-\frac{\beta_3}{2})\csc\frac{\beta_3}{2},\\
        c_3=\eta_{P}\tan\frac{\beta_3}{2}.
  \end{array}}
  \right.
  \end{flalign}

 Now the solution in supersonic regions has been determined; see \eqref{eq:Um} and \eqref{eq:U2}. The corresponding subsonic region $\Omega_{sub}$ is bounded by ${P_1T_2}$, ${T_2T_m}$, ${T_mP_2}$, ${P_2P_3}$ and $P_1P_3$ (see Figure \ref{fg-pro3}). Hence, Problem \ref{prob1} can be reduce to
 \begin{equation}\label{for u1>c1}
     \begin{cases}
     \text{Equation} \eqref{eq:varphi}\quad&\text{in}\quad \Omega_{sub},\\
     \varphi=B_0\quad&\text{on}\quad {P_1T_2}\cup{T_2T_m}\cup{T_mP_2}\cup{P_2P_3},\\
     D\varphi\cdot \bm{\nu}=0\quad&\text{on}\quad P_1P_3, 
     \end{cases}
 \end{equation}
 where $\bm{\nu}$ is the exterior unit normal to $P_1P_3$. Since \eqref{eq:varphi} is rotation-invariant and of reflection symmetry with respect to the axes, we can reflect the region $\Omega_{sub}$ with respect to the boundary $P_1P_3$ to get a region $\Omega^{ref}_{sub}$, and reduce the problem \eqref{for u1>c1} to a Dirichlet problem:
 \begin{equation}\label{for u1>c1 dir}
     \begin{cases}
     \text{Equation} \eqref{eq:varphi}\quad&\text{in}\quad \Omega^{ref}_{sub},\\
     \varphi=B_0\quad&\text{on}\quad \partial\Omega^{ref}_{sub}.\\
     \end{cases}
 \end{equation}
The existence and regularity of solution to \eqref{for u1>c1 dir} can be obtained by applying the result of \cite[Theorem 6.1]{Serre09}. 
 
 From the above discussion, the case $(ii)$ of Theorem \ref{thm: Main2} has been proved.

\section{shock diffraction by a convex cornered wedge}\label{sec:5} 
  We consider the initial data \eqref{initial condition} separated by a vertical shock above the wedge \eqref{wedge} (see \Cref{fg-shock diffraction}). The shock diffraction occurs when the shock passes through the wedge along the positive $x$-axis. This problem can be described mathematically as follows.

  \begin{problem}\label{pro-shock diffraction}
   We wish seek a solution of system \eqref{eq:Euler system}--\eqref{eq:Chaplygin gas} above the wedge \eqref{wedge} with the initial data \eqref{initial condition} satisfying $u_1>0$, and the slip boundary condition \eqref{slip condition}.
  \end{problem}
  
  \begin{figure}[H]
	\centering
	\begin{tikzpicture}[smooth, scale=0.7]
	\draw  [-latex] (8,3)--(9,3) node [right] {\footnotesize$x$};
        \draw  [-latex]  (8.5,2.5)--(8.5,3.5)node [above] {\footnotesize$y$};
        \draw  (4.5,3)--(4.5,6);
        \node at (4.5,6)[above]{\footnotesize Shock};
        \node at (4.5,3) [below] {\footnotesize$O$};
        \node at (3.5,2) {$W$};
        \draw  [dashed] (4.5,3)--(6.5,3);
	\draw  (4.5,3)--(8,1);
        \draw [draw=gray, fill=gray, fill opacity=0.6](4.5,3)--(8,1)--(1,1)--(1,3)--(4.5,3);
        \node at (2.5,4.5){\footnotesize$U_1=(u_1>0,0,c_1)$};
        \node at (6.5,4.5){\footnotesize$U_0=(0,0,c_0)$};
        \draw (5,2.75)arc (315:370:0.3);
        \node at (5.75,2.75){\footnotesize$\theta_0>0$};
	\end{tikzpicture}
	\caption{The problem of shock diffraction.} 
	\label{fg-shock diffraction}
  \end{figure}
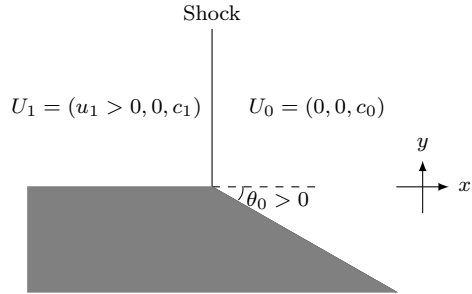

From \eqref{eq:Um}, it follows that if the initial data satisfy $c_0=u_1+c_1$ and $u_1>0$, then the intermediate state $U_m$ coincides with the initial state $U_1$, and \Cref{prob1} can be reduced to \Cref{pro-shock diffraction}. Then, we can determine the pressure wave structure as shown in \Cref{shock diffraction 1,shock diffraction 2}, which are simplified from \Cref{fg-pro2} and \Cref{fg-pro3}, respectively. Also, note here that for the case $c_1<u_1<c_0+c_1$, as the intermediate state $U_m$ vanishes, the points $P,T_1,T_2$ coincide with $T_m$, and the center $O_1$ coincides with $O_m$, and the state $U_3$ coincides with $U_2$.

\begin{figure}[H]
	\centering
	\hspace{0.45cm}
	\subfigure[The case $u_1<c_1$.]{
		\begin{tikzpicture}[smooth, scale=0.53]
        %\draw [step=1,help lines] (0,0) grid (14, 7);
        \draw  [dashed] (4.5,3)--(8.5,3);
        \draw  [dashed] (8,3) arc (0:180:2.5);
        \draw  (7.36,1.38) arc (315:360:2.2);
        \fill (7.36,1.38) circle (2pt);%画点
        \node at (7.4,1.38) [right] {\footnotesize$P_3$};
	  \draw  (4.5,3)--(8,1);
        \draw [draw=gray, fill=gray, fill opacity=0.6](4.5,3)--(8,1)--(1,1)--(1,3)--(4.5,3);
        \draw (5,2.75)arc (315:370:0.3);
        \fill (5.5,3) circle (2pt);%画点
        \node at (5.5,3) [above] {\footnotesize$O_m$};
        \fill (4.5,3) circle (2pt);%画点
        \node at (4.5,3) [below] {\footnotesize$O$};
        \fill (3,3) circle (2pt);%画点
        \node at (3,3) [below] {\footnotesize$P_1$};
        \fill (8,3) circle (2pt);%画点
        \node at (8,3) [below right] {\footnotesize$P_2$};
        \draw  (8,3)--(8,8.5);
        \node at (8,5) [right] {\footnotesize$L_{m0}$};
        \node at (5,5) [right] {\footnotesize$\Gamma_{sonic}$};
        \node at (8,2) [right] {\footnotesize$\Gamma_{shock}$};
		\end{tikzpicture}
		\label{shock diffraction 1}
	}
	\subfigure[The case $c_1<u_1<c_0+c_1$.]{
		\begin{tikzpicture}[smooth, scale=0.53]
		%\draw [step=1,help lines] (0,0) grid (14, 7);
	%\draw  [-latex] (11,3)--(12,3) node [right] {\footnotesize$\xi$};
        %\draw  [-latex]  (11.5,2.5)--(11.5,3.5)node [above] {\footnotesize$\eta$};
        \draw  [dashed] (2.2,3)--(9.5,3);
        \draw [draw=gray, fill=gray, fill opacity=0.6](2.2,3)--(10,0.5)--(1.8,0.5)--(1.8,3)--(2.5,3);
       \draw [thick,domain=0:133] plot ({6.25+2.75*cos(\x)}, {3+2.75*sin(\x)});
       %\draw [dashed,domain=115:180] plot ({6.25+2.75*cos(\x)}, {3+2.75*sin(\x)});
       \fill (6.25,3) circle (2pt);%画点
       \node at (6.25,3) [below] {\footnotesize$O_m(O_1)$};
       \draw [thick,domain=0:-21.4] plot ({3.5+5.5*cos(\x)}, {3+5.5*sin(\x)});
        \fill (9,3) circle (2pt);%画点
        \node at (9,3) [below right] {\footnotesize$P_2$};
        \fill  (8.61,1.0) circle (2pt);%画点
        \node at (8.61,1.1) [right] {\footnotesize$P_3$};
        \draw  (9,3)--(9,8);
        \node at (9,5) [right] {\footnotesize$L_{m0}$};
        \fill (2.2,3) circle (2pt);%画点
        \node at (2.2,3) [below] {\footnotesize$O$};
        \draw  (4.35,5)--(2.2,3);
        \fill (4.35,5) circle (2pt);%画点
        \node at (4.35,5) [left] {\footnotesize$T_m$};
        \draw  [dashed] (4.35,5)--(8,1.15);
        \fill (8,1.15) circle (2pt);%画点
        \node at (8.1,1.1) [above] {\footnotesize$O_2$};
       \draw (4.35,5)arc (133:160:5.7);
       \fill (2.85,2.76) circle (2pt);%画点
       \node at (2.85,2.76) [below] {\footnotesize$P_1$};
		\end{tikzpicture}
		\label{shock diffraction 2}
	}	
	\caption{The problem of shock diffraction.}
\end{figure}
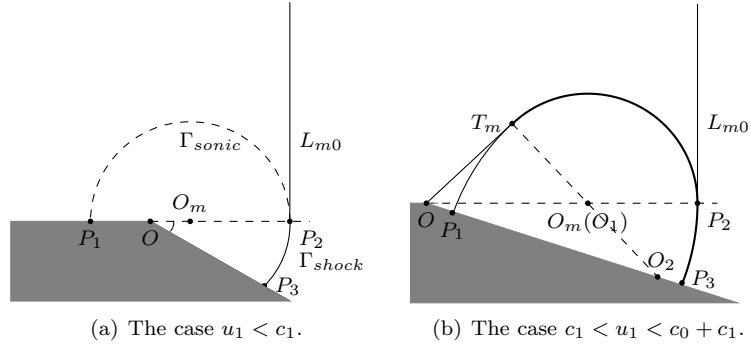

Applying the discussion in \Cref{sec:3,sec:4}, we have
\begin{theorem}\label{thm: shock diffraction}
For \Cref{pro-shock diffraction}, we have the following conclusion:

$(i)$ if $u_1<c_1$, then there does not permit a global Lipschitz solution similar with \Cref{prob2} under the assumption that the flow is continuous across the sonic boundary $\Gamma_{sonic}$.

$(ii)$ if $c_1<u_1<c_0+c_1$ and $\theta_0<\arcsin({c_1/u_1})$, then \Cref{pro-shock diffraction} exists a piecewise smooth solution.
\end{theorem}

We remark here that for the above case $c_1<u_1<c_0+c_1$, since there is no intersection of different waves, we only need $\theta_0<\arcsin({c_1/u_1})$ to ensure that the shock $P_2P_3$ can intersect with the wedge.

\appendix
\section{The critical case $u_1=c_1$}\label{Appendix1}
Compared to the case $u_1<c_1$, the boundary $\Gamma_{+}$ (i.e., ${OP_1}$ in \Cref{fg-pro2}) shrinks to the origin $O$ when $u_1=c_1$. Then, the pressure wave structure is shown in \Cref{fg-pro4}. 

   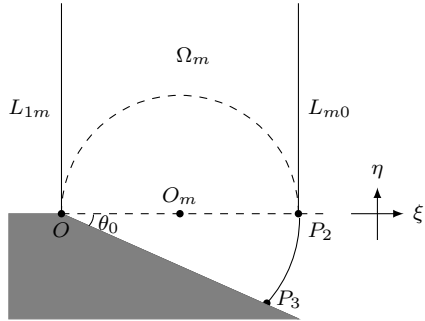
\begin{figure}[H]
	\centering
	\begin{tikzpicture}[smooth, scale=0.7]
        %\draw [step=1,help lines] (0,0) grid (9, 7);
	\draw  [-latex] (9,3)--(10,3) node [right] {\footnotesize$\xi$};
        \draw  [-latex]  (9.5,2.5)--(9.5,3.5)node [above] {\footnotesize$\eta$};
        %\draw  (4.5,3)--(4.5,6);
        \draw  [dashed] (3.5,3)--(8.5,3);
        \draw  [dashed] (8,3) arc (0:180:2.25);
        \draw  (7.4,1.3) arc (318:360:2.42);
        \fill  (7.4,1.3) circle (2pt);%画点
        \node at (7.4,1.38) [right] {\footnotesize$P_3$};
	  \draw  (3.5,3)--(8,1);
        \draw [draw=gray, fill=gray, fill opacity=0.6](3.5,3)--(8,1)--(2.5,1)--(2.5,3)--(3.5,3);
        %\node at (2.5,4.5){\footnotesize$U_1=(u_1,0,c_1)$};
        %\node at (6.5,4.5){\footnotesize$U_0=(0,0,c_0)$};
        \draw (4.02,2.77)arc (315:360:0.3);
        \node at (4.4,2.8){\footnotesize$\theta_0$};
        \fill (5.75,3) circle (2pt);%画点
        \node at (5.75,3) [above] {\footnotesize$O_m$};
        \fill (3.5,3) circle (2pt);%画点
        \node at (3.5,3) [below] {\footnotesize$O$};
        %\fill (3,3) circle (2pt);%画点
        \draw  (3.5,3)--(3.5,7);
        \node at (3.5,5) [left] {\footnotesize$L_{1m}$};
        \fill (8,3) circle (2pt);%画点
        \node at (8,3) [below right] {\footnotesize$P_2$};
        \draw  (8,3)--(8,7);
        \node at (8,5) [right] {\footnotesize$L_{m0}$};
        \node at (5.5,6) [right] {\footnotesize$\Omega_m$};
	\end{tikzpicture}
	\caption{The case $u_1=c_1$.}
	\label{fg-pro4}
  \end{figure}

 By abuse of notations, we continue to write $\Gamma_{sonic}$, $\Gamma_{shock}$ and $\Gamma_{+}$, respectively, for the arcs $OP_2$, $P_2P_3$ and the line $OP_3$. Also, let $\Omega_{sub}$ be the subsonic bounded by $\Gamma_{sonic}$, $\Gamma_{shock}$ and $\Gamma_{+}$. Then Problem \ref{prob1} can be reduced to a boundary value problem for $\varphi$ as follows:
 \begin{equation}\label{for u1=c1}
     \begin{cases}
     \text{Equation} \eqref{eq:varphi}\quad&\text{in}\quad \Omega_{sub},\\
     \varphi=B_0\quad&\text{on}\quad \Gamma_{sonic}\cup\Gamma_{shock},\\
     D\varphi\cdot \bm{\nu}=0\quad&\text{on}\quad \Gamma_{+}. 
     \end{cases}
 \end{equation}
 Thanks to the feature of equation \eqref{eq:varphi}, we can reflect the region $\Omega_{sub}$ with respect to $\Gamma_{+}$ to get a region $\Omega^{ref}_{sub}$. Obviously, the boundary of $\Omega^{ref}_{sub}$ is non-convex. Hence, the existence of the solution in $\Omega^{ref}_{sub}$ is hard to obtain. In fact, the angle of $O$ formed by $\Gamma_{+}$ and $L_{1m}$ is $\frac{\pi}{2}+\theta_0$. Then, by the specular reflection, it is equivalent to a corner with angle $\pi+2\theta_0$ on the boundary of $\Omega^{ref}_{sub}$. It follows from this fact and the result in \cite{Grisvard85} that we can not expect to establish a Lipschitz estimate of $\varphi$ on the boundary, which indicates that the method used in \cite{Serre09} is not valid here.

%\backmatter
\section*{Declarations} 
\textbf{Dual Publication:} The manuscript was done by Long alone. The results in this manuscript have not been published elsewhere.

\vspace{2mm}

\textbf{Funding:} This work was supported in part by National Natural Science Foundation of Hubei Province of China (Grant number 2024AFB007) and Science and Technology Research Project of Education Department of Hubei Province (Grant number D20232901).

%\vspace{2mm}

%\textbf{Data Availability Statement:} Data sharing is not applicable to this article because no datasets were generated/analyzed during the preparation of the paper.

\vspace{2mm}

\textbf{Conflict of interest:} The author has no relevant financial or non-financial interests to disclose.

%\section*{Acknowledgments} 
%This work was supported in part by National Natural Science Foundation of Hubei Province of China (Grant number 2024AFB007) and Science and Technology Research Project of Education Department of Hubei Province (Grant number D20232901).

%%%----------------------------------------------------------------------------

\end{document}